%% file: main.tex
\documentclass[12pt]{article}
\usepackage[top=1in, bottom=1in, left=1in, right=1in]{geometry}

\usepackage{amsthm,soul, amsfonts}
\usepackage{amsmath, amssymb}
\usepackage{authblk,array,setspace} 
\usepackage{natbib}
\usepackage[pdftex,colorlinks=true,linkcolor=blue,citecolor=blue,urlcolor=blue,bookmarks=false,pdfpagemode=None]{hyperref}
\usepackage{enumerate}

\usepackage{tikz}
\usepackage{adjustbox}
\usepackage[linesnumbered,ruled]{algorithm2e}

\usetikzlibrary{backgrounds}

\usetikzlibrary{3d,calc}
\usepackage{tikz-3dplot}
\usetikzlibrary{arrows,positioning,fit,shapes.misc} 
\usetikzlibrary{shapes.geometric}

\usepackage[absolute,overlay]{textpos}

\newcommand\marktopleft[1]{%
    \tikz[overlay,remember picture] 
        \node (marker-#1-a) at (0,1.5ex) {};%
}

\newcommand\markbottomright[1]{%
    \tikz[overlay,remember picture] 
        \node (marker-#1-b) at (0,0) {};%
    \tikz[overlay,remember picture,thick,dashed,inner sep=3pt]
        \node[draw,rectangle,fit=(marker-#1-a.center) (marker-#1-b.center)] {};%
}

\input{macros}

\usepackage{amsthm}
\newtheorem{theorem}{Theorem}[section]

\newtheorem{definition}[theorem]{Definition}

\newcommand{\Ave}{{\rm Ave}}
\newcommand{\acc}[2]{\mathrm{\zeta}[#1, #2]}
\newcommand{\accX}[2]{\mathrm{\zeta}_X[#1, #2]}
\newcommand{\J}[2]{\mathrm{J}[#1, #2]}

\newcommand{\h}{\mathsf{L}}
\newcommand{\hi}[1]{\h_{#1}}

\begin{document}
\pagestyle{empty}

\title{
Propensity score methodology in the presence of network entanglement between treatments
\protect\thanks{Panos Toulis is an Assistant Professor of Econometrics and Statistics, and a John E. Jeuck Faculty Fellow, at the University of Chicago Booth School of Business (\href{mailto:panos.toulis@chicagobooth.edu}{panos.toulis@chicagobooth.edu}).  Alexander Volfovsky is an Assistant Professor of Statistical Science at Duke University (\href{mailto:alexander.volfovsky@duke.edu}{alexander.volfovsky@duke.edu}).
Edoardo M.~Airoldi is an Associate Professor of Statistics at Harvard University (\href{mailto:airoldi@fas.harvard.edu}{airoldi@fas.harvard.edu}). 
This work was partially supported 
 by the National Science Foundation under grants 
  CAREER IIS-1149662 and IIS-1409177,
 by the Office of Naval Research under grants 
  YIP N00014-14-1-0485 and N00014-17-1-2131, 
 by a Shutzer Fellowship and an Alfred Sloan Research Fellowship to Edoardo M.~Airoldi and by a NSF MSPRF for Alexander Volfovsky (DMS-1402235).
The authors are grateful to Daniel Sussman, Dean Eckles, Elizabeth Ogburn, David Choi and Hyunseung Kang for helpful discussions.}}
\author{Panos Toulis}
\affil{Booth School of Business, University of Chicago}
\author{Alexander Volfovsky} 
\affil{Department of Statistical Science, Duke University}
\author{Edoardo M. Airoldi}
\affil{Department of Statistics, Harvard University}
\maketitle

\newpage
\input{abstract}

\newpage
\singlespacing
\tableofcontents
\normalsize
\doublespacing

\newpage
\pagestyle{plain}
\setcounter{page}{1}

\section{Introduction}
\label{sec:intro}
In causal inference we are typically interested in evaluating the effects of a treatment 
that can be applied {\em individually} on the experimental units. However, when the units form networks (e.g., social or professional networks) the treatment is frequently applied to pairs, or more generally, on groups of connected units. Thus, the treatments of individual units are {\em entangled}. 
Notions of entanglement pervade many fields: game theory has long emphasized the central role that social and professional networks have in labor market outcomes and mobility \citep{montgomery1991social,montgomery1992job,podolny1997resources,calvo2004effects}; labor economics studies the effect of the density of a professional network on knowledge diffusion and innovation \citep{dahl2002embedded,kim2005labor,agrawal2006gone,oettl2008international,granovetter2005impact,topa2001social,kaiser2011labor}; marketing has recently been concerned with understanding the financial and social value of online friendships \citep{ellison2007benefits,hanna2017positive,hobbs2016online,manchanda2015social,zhang2014working}.
Such {\em treatment entanglement} presents new methodological challenges that have not been addressed in the literature despite the growing interest in evaluating treatment effects on networks of connected units. 

The typical concern when there is a network connecting the experimental units is that 
 of interference~\citep{cox1958planning, rubin1974estimating}, which means that
the individual outcome  of a unit may depend on the treatment assignment of other units. 
When interference is present it is no longer tenable that units have only two potential outcomes (as is commonly assumed in classical analysis) depending on whether they are treated or not, but may have more depending on who else is treated.
A frequent assumption to simplify analysis in the network setting  is that of 
neighborhood interference, i.e., that a unit $i$'s outcome may be affected only by the treatment of units that are connected to $i$ in the network~\citep{sussman2017elements}. Under such assumptions on the topology of interference, a line of literature has emerged on developing mainly randomization-based methodologies to test hypotheses on causal effects, including typical primary effects and spillover effects from interference~\citep{rosenbaum2007interference, hudgens2008toward,toulis2012estimation,toulis2013estimation,aronow2013estimating, bowers2013reasoning, basse2015optimal, karwa2016bias, choi2016estimation, sussman2017elements, eckles2017design,jagadeesan2017designs, basse2017exact}.

%

%
In the context of observational studies, where treatment assignment cannot be manipulated, there has been relatively less progress, which has mostly focused on model-based identification of causal effects~\citep{manski1993identification, bramoulle2009identification, manski2013identification, Forastiere:2016aa}. Results in this line of work typically characterize conditions in the underlying network that ensure identification of the model parameter of causal interest.
A key complication in observational studies in contrast to randomized experiments is reasoning about the exact mechanism that produces interference. 
For example, one mechanism could be that units that are connected in the network might have similar outcomes to treatment because they are simply similar units to begin with (homophily); 
another could be that a unit's outcomes may be affected by the outcomes of their neighbors in previous time points (contagion). The combined causal effects from homophily and contagion are empirically indistinguishable without additional assumptions~\citep{shalizi2011homophily}. 

The problem of entanglement, on the other hand, has received no attention in the literature.
To illustrate, Figure~\ref{figure:intro} depicts the type of datasets this paper is concerned with. The experimental units, indexed by $i$, form a network $G^-$ at time $t^-$, and the network evolves to $G^+$ at time $t^+$. The individual treatment $Z_i$ for each unit $i$ is defined as a function of the change from $G^-$ 
to $G^+$, which we denote as $Z_i = f_i(G^-, G^+)\in\mathbb{R}$. For instance, 
$Z_i$ could be the number of new connections $i$ makes from $\Gm$ to $\Gp$. Outcomes $Y_i\in\mathbb{R}$ are measured for each unit $i$ at $t^+$.
Our goal is to measure the causal effect of treatment $Z$ on outcome $Y$ through the 
framework of potential outcomes, which we detail in the following section.

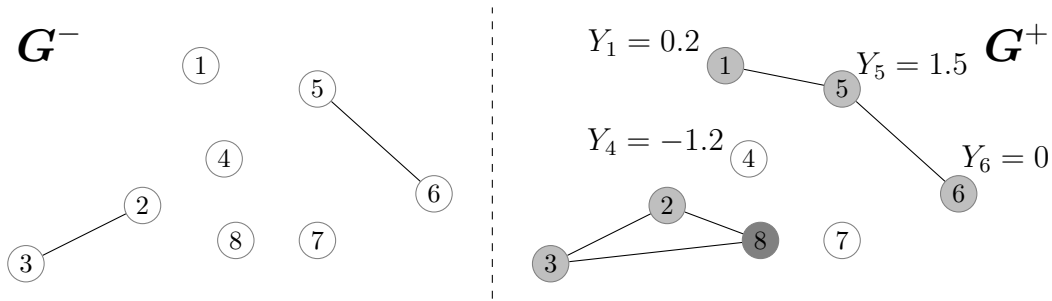
\begin{figure}[t!]
\centering
\begin{tikzpicture}[scale=1.55]
\drawunit{3.5}{2}{white}{1}
\drawunit{3}{0.8}{white}{2}
\drawunit{2}{0.3}{white}{3}
\drawunit{3.5+0.2}{1.2}{white}{4}
\drawunit{3.5 + 1.0}{1.8}{white}{5}
\drawunit{3.5 + 2}{0.9}{white}{6}
\drawunit{2.5 + 2}{0.5}{white}{7}
\drawunit{2.5 + 1.3}{0.5}{white}{8}
\draw  (2) -- (3);
\draw  (5) -- (6);

\draw [dashed] (6, 0) -- (6, 2.5);

\drawunit{4.5 + 3.5}{2}{lightgray}{1}
\drawunit{4.5 +  3}{0.8}{lightgray}{2}
\drawunit{4.5 + 2}{0.3}{lightgray}{3}
\drawunit{4.5 + 3.5+0.2}{1.2}{white}{4}
\drawunit{4.5 +  3.5 + 1.0}{1.8}{lightgray}{5}
\drawunit{4.5 +  3.5 + 2}{0.9}{lightgray}{6}
\drawunit{4.5 +  2.5 + 2}{0.5}{white}{7}
\drawunit{4.5 + 2.5 + 1.3}{0.5}{gray}{8}
\draw  (2) -- (3);
\draw  (5) -- (6);
\draw (5) -- (1);
\draw (3) -- (8);
\draw  (2) -- (8);
\node at (2.2, 2.2) {\Large $\boldsymbol{G}^{-}$};
\node at (10.5, 2.2) {\Large $\boldsymbol{G}^{+}$};
\node at (9.6, 2) {$Y_5=1.5$};
\node at (7.3, 2.2) {$Y_1=0.2$};
\node at (7.4, 1.35) {$Y_4=-1.2$};

\node at (10.4, 1.2) {$Y_6=0$};
\end{tikzpicture}
\caption{A network $\Gm$ among 8 units evolves 
during the treatment period to a new network $\Gp$. 
Treatment of a unit is the number of new connections the unit makes 
from $\Gm$ to $\Gp$.
Shades indicate number of new connections: white = none
(units 4 and 7), lightgray = one 
(units 1, 2, 3, 5, 6), and gray = two (unit 8).
Outcomes $Y_i$ are measured on each unit $i$, potentially 
with covariate information $X_i$. 
%
}
\label{figure:intro}
\end{figure}

%
%

In this paper, we assume that the evolution from $G^-$ to $G^+$ is endogeneous, and thus the evaluation of causal effects may be confounded with the units' covariates. Information such as age, gender, occupation, or location may affect how the network evolves.
For example, in Figure~\ref{figure:intro} the situation could be such that 
units that form  more professional connections from $t^-$ to $t^+$ improves their monthly wages. It would be tempting to associate making new connections to improved wages, but in doing so
we ignore that being more sociable confounds both making more new professional connections and having better professional outcomes.

One general method to avoid such biases from units self-selecting into 
treatment is to use propensity score methods~\citep{rosenbaum1983central, rosenbaum1984reducing, heckman1990varieties, rubin2006estimating}. 
The idea is to model, for every unit $i$, the regression functions $E(Z_i | X_i)$ conditional on $i$'s covariates $X_i$, and then compare outcomes of units with 
similar regression functions.~When $Z_i$ is binary (with $Z_i=1$ denoting the 
treatment, by convention) the regression of $Z_i$ on $X_i$ is known as the propensity score of unit $i$ because it 
is equal to the propensity, $p(Z_i=1|X_i)$, of a unit with 
covariates $X_i$ to receive the treatment.
\citet{rosenbaum1983central} showed that conditional on 
the propensity scores the bias due to the aforementioned confounding is mitigated, and thus we can compare within-group outcomes between treated and control units to obtain estimates of the causal effect of the treatment.%

However, the classical propensity score methodology
cannot be directly applied in our setting. One complication can be that the treatment is multivalued and not binary,
and there is limited work so far in generalizing propensity scores to non-binary treatments \citep{imbens2000role, hirano2004propensity}.
More importantly, in modeling $E(Z_i |X_i)$ the classical methodology tacitly assumes that the treatment is applied individually on each unit $i$. However, when treatment is entangled, estimation
of $E(Z_i|X_i)$ individually on each unit $i$ may be biased because the individual treatments  depend on the network evolution from 
$G^-$ to $G^+$, which is a population quantity.
%
In the example of Figure~\ref{figure:intro}, when a unit makes one connection this immediately implies that another unit makes a connection.
The correct approach would therefore be to model the propensity score of unit $i$ taking into account information from every other unit $j$  that could connect to $i$ during the evolution of $\Gm$. By ignoring treatment entanglement the classical methodology may provide causal estimates that are severely biased depending on how different these two models are.

In this context, we extend the 
classical propensity score methodology to settings with treatment entanglement. 
Our theoretical contribution relies on a novel similarity function that 
measures how different  two propensity score models are in terms 
of the causal estimates they produce.  
Under specific  assumptions, we show that typical linear propensity score models 
may be severely misspecified when there is treatment entanglement, even in simple settings where the individual treatment 
$Z_i$ depends on covariates of a single dyad of units $(i, j)$.
Our methodological contribution includes an algorithm that takes into account the uncertainty about the observed network through an appropriate network evolution model, and produces improved estimates of the individual propensity scores 
by averaging over probable network evolution trajectories.
%

%

%

\subsection{Paper organization} 
The rest of this paper is organized as follows. 
In Section \ref{sec:prelims} we introduce the necessary notation for causal inference using the potential outcomes framework. 
We focus on the setting of observational studies and provide a brief overview of classical propensity score methodology in Section~\ref{sec:subclassification} and the subclassification method for 
estimating causal effects.
In Section~\ref{sec:entanglement_def} we provide the formal definition of treatment entanglement and in Section~\ref{sec:challenges} we describe the statistical challenges in settings with entanglement, and show that classical propensity score methods are generally unjustified in such settings as they may lead to significant biases of estimated treatment effects.
%
In Section~\ref{sec:methodology} we argue that the subclassification similarity between propensity score models provides a reliable measure  
of causal estimation bias of one model 
assuming the other model is the true one.
To avoid such bias we propose a methodology 
that averages over the uncertainty in the network evolution that determines entanglement, and thus produces a model that has high
subclassification similarity to the true one.
Our theoretical analysis begins in Section~\ref{sec:theory} where we present results on a novel subclassification similarity function 
between two propensity score models. In Section~\ref{sec:inner} and Section~\ref{sec:dyadic} 
we focus on well-known network models. Our analysis shows that classical propensity score methodology may be severely biased as it has minimal subclassification similarity with the aforementioned network models. %
In Section~\ref{sec:discussion} we make some concluding remarks, 
including connections of entanglement to interference and selection on observables.
\section{Causal inference with entangled treatments}
\label{sec:entanglement}
In this section, we develop the concept of treatment entanglement, where 
treatment is conceptualized as the evolution of a network between units. Even though the concept of entanglement is more general we 
focus on the network aspect because it allows to develop all necessary concepts and theory using existing language of graph theory and network analysis.

\subsection{Preliminaries on propensity score methodology} 
\label{sec:prelims}
%
Consider $N$ experimental units, indexed by $i$, such that  
every unit $i$ is either treated, denoted by $Z_i=1$, or is
in control ($Z_i=0$). The binary vector $Z = (Z_i)$ is the population treatment assignment.
The potential outcome for unit $i$ under population treatment $Z$ is denoted by $Y_i(Z)\in\mathbb{R}$.
Assuming that the treatment of individual
$i$ only affects the outcome of $i$ and that there are 
no hidden versions of the treatment---collectively known as stable unit-treatment value assumption (SUTVA) in statistics \citep{rubin1980comment}, or individualistic treatment response (ITR) in the economics literature \citep{manski2013identification}---the potential outcome of every unit $i$ can take only two 
possible values~\citep{neyman1923}, namely, $Y_i(1)$ when $i$ is treated, and $Y_i(0)$ when $i$ is in control.
For each unit $i$ only one potential outcome can be observed in the 
experiment, and is denoted by $Y_i\in \{Y_i(0), Y_i(1)\}$. 

The causal effect of the treatment relative to control is usually captured by the average treatment effect (ATE, \cite{imbens2015causal}), 
$\tau = \Ex{Y_i(1) -Y_i(0)},$
where the expectation is over an assumed superpopulation of units that are randomized to treatment or control.
Under SUTVA, it is easy to estimate the ATE in a randomized experiment from the difference in sample means between 
treated and control units:
\begin{align}
\label{eq:est}
\hat\tau=\frac{\sum_i Z_i Y_i}{\sum_i Z_i}-\frac{\sum_i (1-Z_i)Y_i}{\sum_i (1-Z_i)}.
\end{align}
When treatment is randomly assigned $\hat\tau$ is an unbiased estimator of $\tau$, and variance estimates for $\hat\tau$ are available~\citep{neyman1923}.
In this context, resampling-based procedures are possible to test
for hypotheses on treatment effects, including 
recent extensions to settings with interference
~\citep{rosenbaum2007interference, hudgens2008toward,toulis2013estimation, aronow2013estimating, eckles2017design,basse2017exact}.

However, when units cannot be randomized to treatment, as is frequent in practice, the treated and control 
units may not be directly comparable because the difference in their treatment status reflects differences in their characteristics. In such cases $Z_i$ typically 
depends on unit-level covariates, $X_i$, such as age, location, or socioeconomic status, and simple estimators, such as $\hat\tau$ in Equation~\ref{eq:est}, are biased because  the units that are treated 
may have different covariates $X_i$ than the units in control. 
%
To resolve such biases  one widely-applied methodology is to compare the outcomes of 
units that have similar propensity to select into 
treatment.
Classical propensity score methods first estimate the propensity score, which is the regression function 
\begin{align}
\label{eq:propensity}
e(X_i)= E(Z_i | X_i),
\end{align}
and then use the propensity scores to control for selection bias. 
The main idea relies on the seminal result by~\citet{rosenbaum1983central} who showed that conditional on $e(X_i)$ the treatment indicator $Z_i$ is independent of $X_i$. Therefore, comparing outcomes of treated and control units that have similar propensity scores can yield valid causal inference.

There are different ways in which we can leverage propensity scores for such 
comparisons. One approach uses matching units with similar propensity scores~\citep{cochran1973controlling, rubin1978using, raynor1983caliper, rosenbaum2002observational, dehejia2002propensity, stuart2010matching}, 
and reweighing unit outcomes based on the propensity score~\citep{hirano2001estimation, rubin2001using, lunceford2004stratification}, 
such as the inverse propensity weight estimator (IPW).
In this paper,  we focus on subclassification as it generalizes the other two  methods: matching is point subclassification~\citep{stuart2010matching}, and an IPW estimator is a limit of subclassification estimators \citep{rubin2001using}. 
Subclassification is also robust to scaling and, more generally, to monotone transformations of the propensity score, which provides with a robust measure of propensity score misspecification and bias~(see Section~\ref{sec:theory}).

\subsection{Subclassification on the propensity score}
\label{sec:subclassification}
%
%
%
Subclassification on the propensity score has several variations but 
typically proceeds as follows. Let $K$ be a fixed number of classes; for example, \citet{rosenbaum1983central} suggest $K=5$ or $10$ based 
on normality arguments.
Let $Q_1,\dots,Q_K$, denote $K$ successive quantiles of all
estimated propensity scores $\{\hat e_i\}$, and define the variable $R_{ik}^e = \mathbb{I}\{\hat e_i \in Q_k\}$ as the indicator of whether unit $i$ falls in class $k$ under propensity propensity score model $e$.
\citet{rosenbaum1983central} showed that the estimator 
\begin{align}
\label{eq:est_sub1}
\hat\tau_k =\frac{\sum_i R_{ik}^e Z_i Y_i}{\sum_i R_{ik}^e Z_i}-
\frac{\sum_i R_{ik}^e(1-Z_i)Y_i}{\sum_i R_{ik}^e(1-Z_i)}
\end{align}
is a valid estimate of the causal effect $\tau$, for every class $k$. 
Note that $\hat\tau_k$ is analogous to $\hat\tau$ in Equation~\eqref{eq:est} if we focus only on units in class $k$. Intuitively, within such a class the treatment is as if randomized and so the 
simple difference-in-means estimator is unbiased. 
The within-class estimates can also be combined, say, through the estimator
\begin{align}
\label{eq:est_sub2}
\hat\tau(Y; X, e) = \frac{1}{N} \sum_k N_k \hat\tau_k,
\end{align}
where $n_k = \sum_i R_{ik}^e$ is the size of class $k$, and $N = \sum_k N_k$ is the total number of units. 
In our application the treatment can be multivalued, and thus we need to consider propensity scores of the form $e(l,X_i)=p(Z_i=l|X_i)$, where $l\in\mathbb{N}$ is a natural number denoting, say, the number of new network connections $i$ made in the treatment period. Under the same assumptions as in the classical binary treatment setting we can still subclassify units based on the treatment levels we want to compare. For instance, if we wish to compare the outcomes of a unit making two connections against making only one connection, we could subclassify units based on pair values $\{(e(2, X_i), e(1, X_i)\}$ as in~\citep{imbens2000role,hirano2004propensity}, since causal estimation 
will be performed conditional on the units receiving either $Z_i=1$ or $Z_i=2$.

\subsection{Treatment entanglement} 
\label{sec:entanglement_def}
We assume that the treatment 
takes place from time $t^-$ to $t^+$. In particular, the units form 
connections at $t^-$ denoted by a $N\times N$ binary matrix $G^-$, such that $g_{ij}^-=1$ only if units $i$ and $j$ are connected at $t^-$.  Similarly, $G^+$ denotes connections between units at $t^+$.
For simplicity, and without loss of generality, we assume undirected networks, and thus symmetric matrices $G^-, G^+$.
%
%
The individual treatment $Z_i$ on unit $i$ is defined as the number of new connections in the treatment period:
\begin{align}
\label{eq:Zi}
Z_i = d_i(\Gp) - d_i(\Gm),
\end{align}
where $d_i(G) = \sum_j g_{ij}$ denotes the degree of unit $i$ in network $G$.
Such definition allows us to study a broad range of potential interesting policy questions, such as the effect of interactions with peers on individual outcomes~\citep{aral2009distinguishing}, 
or the effect of making new professional connections on labor outcomes~\citep{montgomery1991social}.
We note that the specification of the treatment does not require any assumptions on the presence or absence of interference. Throughout we will make a no interference assumption to focus more clearly on the problem of entanglement, and will discuss extensions to settings with interference in Section~\ref{sec:discussion}.

The important thing to notice about Equation~\eqref{eq:Zi}
is that it implies a constraint between the units' treatments, which may be expressed as:
\begin{equation}
\label{eq:entanglement}
\h(Z)= \h(Z_1, Z_2, \ldots, Z_n) = 0.
\end{equation}
For example, the constraint in Equation~\eqref{eq:Zi} can be written as $Z - (G^+ - G^-) \boldsymbol{1} = 0$, where $\mathbf{1}$ denotes the $N$-component vector of ones. Such constraints imply that there is {\em treatment entanglement}, that is, the individual treatments of units depend on each other.
Function $\h$ is the {\em entanglement function}, and is generally vector-valued.

Furthermore, the individual treatment $Z_i$ of unit $i$ may be expressed more generally as $Z_i=f_i(\Gm,\Gp)$, where function $f_i$ is abstractly defined as a function mapping from network evolution to treatment space. Potential meaningful specifications 
for this function include Equation~\eqref{eq:Zi}, which 
is the specification that we use in this paper.
Alternative definitions could be whether the neighborhood of unit $i$ grew, $f_i(\Gm,\Gp)=\mathbb{I}\{d_i(\Gp)>d_i(\Gm)\}$, where $\mathbb{I}$ is the indicator function, and so on.
Functions $f_i$ are application-specific, and imply the 
generic entanglement constraint $\h(Z) = Z-f(\Gm, \Gp) = 0$, where 
$f$ applies $f_i$ component-wise.
Even more generally, we can consider probabilistic entanglement where rather than requiring $\h(Z) = 0$ we require $P(\h(Z)=0) > 0$.
We explore this setting in Section~\ref{sec:prob_ent}.

%
For an illustration of simple entanglement, consider the population of two units in Figure~\ref{fig:toy_ent}(a).
If unit 1 forms a connection to unit 2, then unit 2 must also form a connection to unit 1, and vice versa. This implies the entanglement constraint 
$
\h(Z_1, Z_2) = Z_1-Z_2=0.
$
 In this setting we can therefore only observe two untreated or two treated units.
 We note that our definition of treatment entanglement is substantially more general than just the network case we consider in this paper. In fact, one of the most commonly studied experimental designs, the completely randomized design (CRD), is entangled. In a CRD where the fraction of treated units is $p$, the entanglement function is  $\hi{}(Z)=Z^\top \mathbf{1}-np = 0$. We note that this particular type of entanglement (while not identified by name) is frequently accounted for in observational studies. Recent work by \citet{branson2017randomization} addresses the differences in causal inference under assumptions of a CRD versus the unentangled Bernoulli trial.

%
%
\begin{figure}[t]
\centering
\begin{subfigure}{.33\textwidth}
\centering
            \begin{tikzpicture}[->,>=stealth',shorten >=1pt,auto,node distance=2cm, thick, column sep=1cm,main node/.style={circle, draw=black, fill=none}]
        \node[main node] (1) {1};
        \node[main node] (2) [right of=1] {2};
        \end{tikzpicture}
        \caption{Pre-treatment network $\Gm$ 
        between two units.}
\end{subfigure}\hfill
\begin{subfigure}{.32\textwidth}
\vspace{10pt}
\centering
        \begin{tikzpicture}[->,>=stealth',shorten >=1pt,auto,node distance=2cm, thick, column sep=1cm,main node/.style={circle, draw=black, fill=none}]
      
        \node[main node] (1) {1};
        \node[main node] (2) [right of=1] {2};
        \end{tikzpicture}
        \caption{Post-treatment network $\Gp$ where neither unit is treated ($Z_1=Z_2=0$).}
      
      
\end{subfigure}\hfill
\begin{subfigure}{.32\textwidth}
\vspace{10pt}
\centering
                  \begin{tikzpicture}[-,>=stealth',shorten >=1pt,auto,node distance=2cm, thick, column sep=1cm,main node/.style={circle, draw=black, fill=none}]
      
        \node[main node] (1) {1};
        \node[main node] (2) [right of=1] {2};
      
         \path[every node/.style={font=\sffamily\small}]
         (1) edge node []{} (2);
        \end{tikzpicture}
        \caption{Post-treatment network $\Gp$ where both units are treated ($Z_1=Z_2=1$).}
      \end{subfigure}

  \caption{Pre- and post-treatment networks with two units. This is the smallest possible example of entangled treatments when treatment is measured as the formation of new edges (e.g., making a friend) in the time interval between $\Gm$ and $\Gp$.}
    \label{fig:toy_ent}
\end{figure}
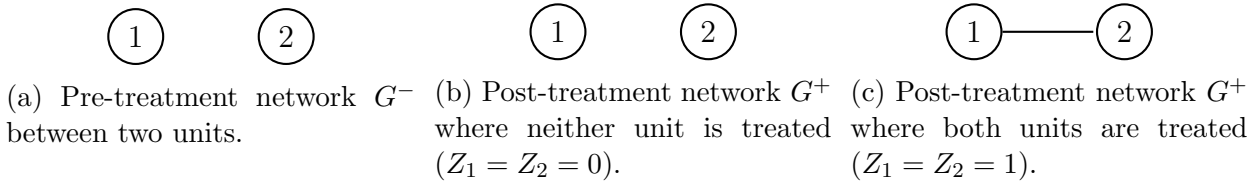

The simple example of Figure~\ref{fig:toy_ent} also allows us to illustrate the failure of classical propensity score methodology under treatment entanglement.
Consider the case in Figure~\ref{fig:toy_ent}(c) where both units are treated ($Z_1=Z_2=1$),
and suppose that unit 1 is a junior employee, unit 2 is its senior supervisor, and both units work in the same company.
A classical propensity score model will try to explain the observed treatment assignment based on covariates that are common between the units; here,
such covariate is the company where both units work. 
The classical methodology does not capture {\em dyadic covariates} that are important for the treatment assignment; here, for example, it could be that 
the dyad covariate of ``difference in seniority'', is the significant factor for edge creation. The covariates in the dyad (i.e., seniority level) may not be significant if taken individually. As dyad covariates are not included and the network structure is ignored, the classical methodology has a limited capacity to model treatments defined by networks of units, which leads to biased estimates of causal effects in observational studies on networks. 
In the following section we give more details on the challenges posed by treatment entanglement, and an approach to address them.
%

%


\subsection{Challenges under treatment entanglement}
\label{sec:challenges}
Consider the entangled treatment in Equation~\eqref{eq:Zi}, 
and suppose that edges cannot be deleted in the treatment period, so that $Z_i \ge 0$.
The classical propensity score methodology models the following regression function, which generalizes the definition in Equation~\eqref{eq:propensity}:
\begin{align}
\label{eq:propensity_multi}
e(l, X_i) = P(Z_i=l|X_i, \Gm),~l\in\mathbb{N}.
\end{align}
The values $e(l, X_i)$ are propensity scores, despite $Z_i$ being non-binary,
because they capture the likelihood that unit $i$ with covariates $X_i$ makes $l$ new connections.
%
As mentioned in Section~\ref{sec:subclassification}, to estimate causal effects using subclassification on the propensity scores of Equation~\eqref{eq:propensity_multi} we can group units with similar values of estimated  
propensity scores (possibly pair-values), compare outcomes within groups, and then aggregate comparisons across groups~\citep{hirano2004propensity, lee2015efficient}.

These procedures ignore dependence between individual treatments due to entanglement. 
For example, in Figure~\ref{fig:toy_ent} the entanglement between units 1 and 2, characterized by  the entanglement constraint $\h(Z_1, Z_2) = Z_1 - Z_2 = 0$, implies the constraint
$e(l, X_1) = e(l, X_2)$, which will generally not be satisfied if the propensity score model of Equation~\eqref{eq:propensity_multi} if fit individually per unit.
The subtle issue here is that when we model the probability that $Z_i=l$ we implicitly model the conditional of $\Gp$ given $\Gm$, since $Z_i$ is a function of both of them by Equation~\eqref{eq:Zi}. Classical methods ignore this implicit relationship between $Z_i$ and $\Gp$ and the constraints it entails.~The proper way to compute the propensity scores is therefore to marginalize over 
the post-treatment network, so that
\newcommand{\Gboth}{\Gm,\Gp}
\begin{align}
\label{eq:PS_marginal}
  p(Z_i=l|\boldsymbol{X}, \Gm) = \int\limits_{f_i(\Gm, \Gp)=l} p(\Gp|\Gm, \X)d\mu(\Gp),
\end{align}
where $\X$ denotes the covariates of all units, and $\mu$ is an 
appropriate probability measure on networks. 
Classical propensity score methods ignore the integration in Equation~\eqref{eq:PS_marginal}, and thus effectively remove any dependence between individual treatments from entanglement.
%
What is needed instead is to model the change
from $\Gm$ to $\Gp$, possibly conditional on covariates, and 
then obtain the true propensity scores of units 
by using the treatment definition $Z_i = f_i(\Gm, \Gp)$, which we can then use 
to estimate the causal effect, say, through subclassification.

The choice of model for the network is complicated and application-specific. Possible choices include simple rewiring models \citep{dietz1988epidemiological}, Temporal Exponential Random Graph Models (a generalization of the ERGM framework, \citet{hanneke2007discrete,hanneke2010discrete}) and dynamic latent space models \citep{sarkar2006dynamic,sewell2015latent,durante2014nonparametric}. When it is reasonable to assume that $\Gm$ and $\Gp\setminus\Gm$ are conditionally independent, one can appeal to the generalizability of latent space models \citep{hoffrafteryetal2002}: one can model $\Gm$ conditional on unit and dyadic covariates and use the fitted model to compute the probability of edges in $\Gp\setminus\Gm$. 


Building from these ideas, in the following section we give a detailed description of our full proposed methodology for causal inference under entanglement.

\section{Methodology} 
\label{sec:methodology}
Here, we present our methodology for adapting the propensity score methodology to cases where there is treatment entanglement. 
Our approach does not rely on a particular definition of treatment or causal estimands. 
We assume that for a treatment definition $Z_i = f_i(\Gm, \Gp)$ the 
potential outcome of every unit $i$ depends only on $Z_i$ and not on the actual value of the population treatment vector $Z$.
Formally, $Y_i(Z) = Y_i(Z')$ if $Z_i = Z'_i$, for every unit $i$ and population assignments $Z, Z'$. This is a form of SUTVA.
This allows us to focus on simple causal estimands of the form 
$\tau_m = E(Y_i(m)-Y_i(m-1))$, which capture the incremental causal effect from having $m-1$ new connections in the treatment period to having $m$.

Our methodology is presented in the form of Algorithm~\ref{algo:ent}. The 
input data include the pre-treatment and post-treatment networks, $\Gm$ and $\Gp$, population covariates denoted by $\X$, the realized population treatment vector $Z$, and the observed outcomes $Y$.
The output is an estimate of the causal effect, which relies on subclassification on the (entanglement-adjusted) propensity scores. The main idea in the procedure, which was also described in Section~\ref{sec:challenges}, 
is in Line 4 where we estimate the propensity scores correctly by sampling over the conditional distribution of $\Gp$, and then applying the treatment definition.


\begin{algorithm}[t!]
\setstretch{1.6}
    \SetKwInOut{Input}{Input}
    \SetKwInOut{Output}{Output}
    \Input{Input data: (i) networks $\Gm,\Gp$, (ii) covariates $\X = \{X_1, X_2, \ldots, X_n\}$, (iii) treatment vector $Z=(Z_1, Z_2, \ldots, Z_n)$, \
    (iv) outcomes $Y = (Y_1, Y_2, \ldots, Y_n)$.}
    \Output{Estimate of treatment effect $\hat\tau_m$ of $\tau_m$.}

    Compute the treatment assignments $Z_i=f_i(\Gm, \Gp)$.

    Let $\Gp|\Gm,\X$ be modeled via $p(\cdot|\theta,\Gm,\X)$---obtain 
    estimate $\hat\theta$.

    Use $\hat\theta$ to sample $\Gp_{(b)}$, for $b=1,\dots,B$, conditional on observed $\Gm$.

    Use samples $\{\Gp_{(b)}\}$ to compute $\hat e(l,X_i) = p(Z_i=l | X_i)$ 
    using the empirical frequencies:
    $$\hat e_{i, l} =  \hat{e}(l, X_i) = \frac{1}{B} \sum_{b=1}^B \mathbb{I}\{f_i(\Gm, \Gp_{(b)})=l\}.$$

    Subclassify units according to pairs 
    $(\hat e_{i, m-1}, \hat e_{i, m})$. That is, units should be grouped together if both pair values are similar.
    
Obtain estimates of $\tau_m$ within classes, and combine 
estimates across classes into final $\hat\tau_m$, following classical propensity
score methodology~(see Equations~\eqref{eq:est_sub1} and~\eqref{eq:est_sub2}).
  
    \Return{$\hat\tau_m$.}
    \caption{
    Estimation of treatment effects accounting for entangled treatment\vspace{5px} }
    \label{algo:ent}
\end{algorithm}

Several steps in our procedure could in fact be tuned by the substantive scientist to better accommodate a particular problem:\begin{itemize}
\item The choice of treatment definition needs to be specific 
in Line 1. Such definition depends on the particular entanglement in the application. Some examples were given in Section~\ref{sec:entanglement}. The choice should take into account 
the constraints in Line 4, where the propensity score to treatment level $l$
depends on the likelihood of observing $\Gp$ such that $f_i(\Gm, \Gp) = l$. A poor choice of $f_i$ might lead to a situation where interesting treatment levels are unlikely to be observed in practice or simulation.

\item The choice of network evolution model in Line 2 and the technique for estimating model parameters $\theta$ are both highly domain-specific and depend on subjective judgement and knowledge. The choices here could be as varied as the Exponential Random Graph Model~\citep{frank1986markov} to the general class of latent space models~\citep{hoffrafteryetal2002}.
Additionally, specific econometric or game-theoretic considerations may be taken into account while modeling  network evolution~\citep{galeotti2006network, jackson2010social, chandrasekhar2011econometrics, graham2015methods}.
\item In our procedure, sampling of $\Gp$ conditional on observed data 
is done via a parametric bootstrap approach. A fully Bayesian approach is also reasonable, where sampling of $\Gp$ could be done via the full predictive distribution of $\Gp$~\citep{gelman2003bayesian}. 
The network model could also be completely nonparametric, as long as sampling of $\Gp$ from its conditional distribution is possible.

\item Lastly, a clustering technique needs to be chosen in Line 5 in order to subclassify units based on a multidimensional propensity score. A natural choice, which we use in the experiments of Section~\ref{sec:experiments}, is $k$-means clustering where 
the pairs $\{(\hat e_{i, m-1}, \hat e_{i, m})\}$ are split in $k$ clusters, and causal estimates are taken within each cluster.
The same approach could be applied using alternative clustering techniques~\citep{friedman2001elements}.
\end{itemize}

Our methodology, as described in Algorithm~\ref{algo:ent}, can therefore admit a wide range of concrete realizations. One common characteristic of all such realizations is the use of 
subclassification on the propensity score for estimation of causal effects, whereas the differentiating factor among them is the actual propensity score model that each realization entails. Therefore, to understand the performance of of our methodology 
it is important to understand how similar its causal inferences are to inferences from the true propensity score model or from a misspecified model that ignored the treatment entanglement.
 In the following section we aim to understand theoretically such 
 similarity by evaluating the similarity in subclassification of experimental units by different 
 propensity score models, and argue empirically that this is an adequate measure of similarity of their respective 
 causal inferences.
 
\section{Theory} 
\label{sec:theory}

Here, we develop theory on the subclassification similarity between propensity score models, particularly between the true propensity score model and a misspecified one. In Sections~\ref{sec:inner} and~\ref{sec:dyadic} we show that such similarity can explain the difference in the causal inferences of network propensity score models with varying degrees of misspecification.

\subsection{Subclassification similarity of propensity score models}
\label{section:similarity}
\citet{rosenbaum1983central} state that poor estimation of the propensity score can introduce bias into the estimation of causal effects. If two propensity score models subclassify the units differently, the causal estimates will also be different. As such, by studying the similarity in subclassification between the true propensity score model and a misspecified model we also learn about the bias of the causal estimates from the misspecified model. The following definition formalizes such similarity.

\begin{definition}[Subclassification similarity]
\label{def:similarity}
Let $K$ be the number of classes, and let $\mathrm{Sym(K)}$ denote the set of all possible permutations of the elements of $\mathcal{K}=\{1, 2, \ldots, K\}$. Let $\sigma(k)$ denote the element of $\mathcal{K}$ that $k\in\mathcal{K}$ maps to under permutation $\sigma$.
Then, the conditional subclassification similarity between models $e$ and $m$ given covariates $\X$ is defined as:
$$
\accX{m}{e} = \max_{\sigma\in\mathrm{Sym(K)}} \frac{1}{N} \sum_{k=1}^K \sum_{i=1}^N R_{ik}^m R_{i\sigma(k)}^e.
$$
The subclassification similarity between models $e$ and $m$ is defined as 
$\acc{m}{e} = \Ex{\accX{m}{e}}$, where the expectation is over the distribution of $\X$.
\end{definition}

{\em Remarks.} 
In simple terms, the subclassification similarity calculates how 
frequently two models subclassify some unit in the same class.
Considering permutations in the definition is crucial because 
two propensity score models may look different in terms of their values but may induce identical subclassification.
For instance, suppose that $m(x) = 1-e(x)$. Then, units that are grouped together under $e$ will also be grouped together under $m$, even though the individual propensity scores from the two models are very different.

Definition~\ref{def:similarity} satisfies several important properties that are desirable for similarity functions: (i) it is normalized to takes value in 
$[0, 1]$; 
(b) it is symmetric, since $\accX{m}{e} = \accX{e}{m}$;
and (c) it is invariant to monotonic transformations. To see this, 
note that $R_{ij}^f = R_{ij}^m$ for every function $f(x) = h(m(x))$, where 
$h$ is monotone. Thus, $\accX{h(m)}{e} = \accX{m}{e}$. 
Invariance to monotone transformations is crucial because subclassification-based causal inference, as typified in Equation~\eqref{eq:est_sub2}, is itself invariant to monotone transformations of the propensity score, since $\hat\tau(Y; X, m) = \hat\tau(Y; X, h(m))$, for any $X, Y$ and monotone  $h$.

\subsection{Approximate subclassification similarity}
\label{section:approximate_similarity}

Unfortunately, $\acc{\cdot}{\cdot}$ is computationally intractable because 
it is defined as an extremum over a complete permutation set. 
Following~\citet{toulis2017volfovsky} we consider an approximate subclassification similarity function that shares the same qualities but is easier to work with, assuming differentiability of propensity scores.
\begin{definition}
\label{def:J}
The approximate subclassification similarity, denoted by $\J{m}{e}$, between propensity score models $m$ and $e$ is defined as follows:
\begin{align}
\label{eq:J}
\J{m}{e} = |E\left(\cos(\nabla m(X_i), \nabla e(X_i))\right) |,
\end{align}
where $\cos(a,b)=\frac{a^\top b}{\|a\|\|b\|}$ is the cosine between vectors 
$a$ and $b$, the gradient is with respect to argument $X_i$, and the expectation is with respect to the distribution of $\X$.
\end{definition}

{\em Remarks.} The geometric intuition behind Definition~\ref{def:J}  is that Equation~\eqref{eq:J} approximates how similarly the contour surfaces of models $m$ and $e$ are oriented in space. For example, at one extreme 
the contour surfaces are perpendicular to each other, meaning that 
the two models subclassify units very differently.
This is reflected in the approximate similarity value being equal to zero (minimum possible value) since the gradients $\nabla e$ and $\nabla m$ are orthogonal to each other. At the other extreme, the contours have the same orientation, 
meaning that the two models subclassify units identically.
This is reflected in the approximate similarity value being equal to one (maximum possible value) since the gradients $\nabla e$ and $\nabla m$ are collinear.
Formally, approximate similarity satisfies all three crucial properties of true subclassification similarity in Definition~\ref{def:similarity}, the most important being invariance under monotone 
transformations~\citep{toulis2017volfovsky}.

%
An analytical result for approximate subclassification similarity can be obtained if we assume a parameterized propensity score model $m(x) = m_\beta(x)= h(x^\top\beta)$, where $h$ is monotone. 
Such linear models are widely used in practice as part of the classical methodology.
The following theorem shows that the approximate subclassification similarity between $m_\beta(x)$ and the true model $e(x)$ can be expressed as a projection of $m_\beta$ onto $e$.

\newcommand{\ngrad}[1]{\nabla_{#1}}
\begin{theorem}
\label{theorem:general}
Let $e(x)$ be the unknown true propensity score model.
Let $m_\beta(x) = h(x^\top\beta)$, which may be misspecified, with $h : \mathbb{R} \to \mathbb{R}$ being a differentiable
monotone function.
For a vector-valued function, $g$, of covariates, let $\ngrad{g}$ denote the expected normalized gradient, that is,   $\ngrad{g}= \Ex{\nabla g(X_i) / || \nabla g(X_i)||}$.
Then,
\begin{align}
\label{eq:theorem_general}
\J{m_\beta}{e} = \left |\ngrad{m_\beta} ^\top \ngrad{e} \right |=
\frac{|\beta^\top \ngrad{e}|}{||\beta||}.
\end{align}
If treatment is binary, then $\beta$ and $e(x)$ are connected through the moment equation:
\begin{align}
\label{eq:theorem_general2}
\Ex{\frac{e(X_i)}{h(X_i^\top\beta)} \frac{h'(X_i^\top\beta)}{1-h(X_i^\top\beta)} X_i} = \Ex{\frac{h'(X_i^\top\beta)}{1-h(X_i^\top\beta)} X_i}.
\end{align}

\end{theorem}

{\em Remarks.}~Theorem~\ref{theorem:general} shows that the 
approximate subclassification 
similarity of the true propensity score model with some linear model depends on the projection of the parameter vector of the linear model
onto the 
expected normalized gradient of the true model.
Geometrically, the expected normalized gradient quantifies the direction of the contours of a function, and so the similarity of propensity score models can also be expressed as 
the inner product of their orientations. When a propensity score model is linear 
the orientation is fixed towards the direction of the parameter, whereas the propensity score contours are hyperplanes perpendicular to this direction.
So, the expected cosine between the two gradients is equal to the cosine of the expectations, which leads to Theorem~\ref{theorem:general}.

The importance of Theorem~\ref{theorem:general} is that it allows us to perform a theoretical analysis of misspecification of classical linear propensity scores when there is entanglement. Such misspecification is a reasonable proxy for bias in causal estimation 
from applying classical methodology in entanglement settings, 
which we confirm empirically in Section~\ref{sec:experiments}.
The use of Theorem~\ref{theorem:general} in practice is illustrated in the theoretical and experimental sections that follow.
In particular, we consider settings where treatment entanglement is defined by well-known network models, and then use Theorem~\ref{theorem:general} to understand whether a classical linear model 
will subclassify units similarly or not with respect to the true model. 
The general picture that emerges is that under typical assumptions 
the classical linear model does not subclassify units similarly to 
the true model.
In Section~\ref{sec:experiments} we show empirically that such dissimilarity directly translates to errors in causal effect estimation.

\subsection{Inner-product network model}
\label{sec:inner}
Consider a setting where there are no connections between units in $\Gm$,
and $\Gp$ is a random network where any two units $i$ and $j$ are 
connected with probability $\expit(a + b X_i^\top X_j)$,
where $a, b$ are constants and $\expit(u) = e^u / (1+e^u)$ is a logistic function. Suppose that $X_i\sim\mathcal{N}(0, \tau^2I)$, independent and identically distributed for every unit $i$.

The treatment for unit $i$ is the difference in degree of $i$ between $\Gp$ and $\Gm$, which implies that $Z_i = d_i(\Gp)$. By independence of individual covariates, it holds that
$$\Ex{Z_i | X_i} = \sum_{j\neq i} \Ex{\expit(a + b X_i^\top X_j) | X_i}
 = (N-1) \Ex{\expit(a + \sigma_i Z)},$$
where $Z$ is standard normal and $\sigma_i^2 = b^2 ||X_i||^2 \tau^2$. Assuming fixed $\tau$ and $b$ without loss of generality, it follows that 
 $E(Z_i | X_i) = r(a,  \sigma_i) $ for some function $r : \mathbb{R} \times \mathbb{R}_+ \to[0, 1]$.  
A key property of $r$ is that it is monotone with respect to $||X_i||$: when $a < 0$ it is monotone increasing with respect to $||X_i||$, and when $a > 0$ it is monotone 
decreasing with respect to $||X_i||$---the proof of this result is given in Appendix~\ref{appendix:approx}.
The marginal propensity score for unit $i$ therefore satisfies 
$e(X_i) = E(Z_i | X_i) = (N-1) \cdot r(a, \sigma_i)$.
The gradient of $e(X_i)$ with respect to vector $X_i$ is
$$
\nabla e(X_i) = (N-1)  \frac{\partial r(a, X_i)}{\partial ||X_i||} \frac{1}{||X_i||} X_i.
$$
Since $X_i\sim\mathcal{N}(0, \tau^2 I)$ it follows that 
$\nabla_e = \Ex{\nabla e(X_i) / ||\nabla e(X_i)||} 
= \mathrm{sign}( \frac{\partial r(a, X_i)}{\partial ||X_i||}) 
\Ex{X_i / ||X_i||} \propto  \Ex{X_i / ||X_i||} = 0$, by symmetry. Monotonicity of $r$ with respect to $||X_i||$ was therefore important in decomposing the expectation into the sign term and the $E(X_i/||X_i||)$ term.

Ignoring the network structure of the problem will likely lead to standard linear propensity score models, such as $m_\beta(X) = \mathrm{expit}(\beta^\top X)$. From Theorem~\ref{theorem:general} we immediately 
conclude that $\J{\ngrad{m_\beta}}{e} = \beta^\top \ngrad{e} / ||\beta|| =0$.
This shows that the linear model is very different in terms of subclassification similarity from the true propensity score model.

This agrees with geometric intuition, as depicted in Figure~\ref{fig:entangled}, because the true propensity score model is 
spherical and the misspecified model is linear. That is, the true model subclassifies units across spheres defined by the norm of the unit covariates, which is depicted as concentric spheres in the figure. 
The linear propensity score model $m_\beta(x)$  subclassifies units across hyperplanes with common orientation defined by vector $\beta$. 
Since any area between two hyperplanes can traverse all spheres, it follows that units in the same class in the subclassification of  $m_\beta(x)$ may actually belong to very different classes in the 
subclassification from $e(x)$. 
Thus, the causal estimate from subclassification on $m_\beta(x)$ will be significantly biased because it includes comparisons of units that are widely different from each other.

\newcommand{\sphere}[2]{
    \draw (-#1,0) arc (180:360:#1cm and #2cm);
    \draw[dashed] (-#1,0) arc (180:0:#1cm and #2cm);
    \draw (0,0) circle (#1cm);
    \shade[ball color=blue!10!white,opacity=0.20] (0,0) circle (#1cm);
}

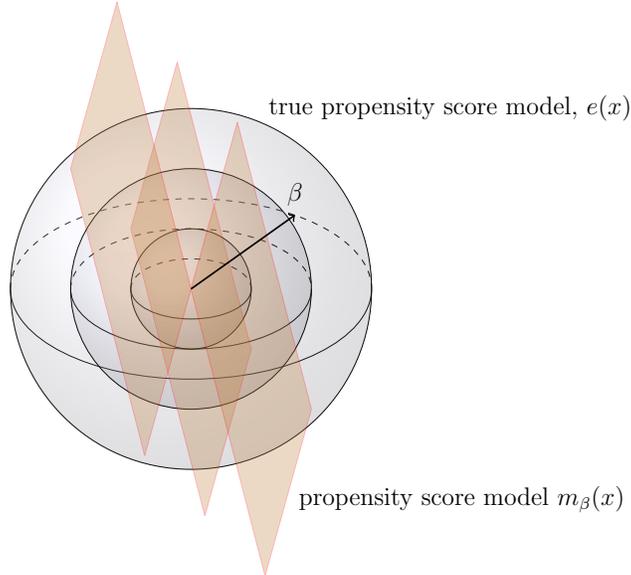
\begin{figure}[t!]
\centering
\scalebox{.8}{
\begin{tikzpicture}
\sphere{3}{1.5}
\sphere{2}{1}
\sphere{1}{0.5}
   \draw[fill=brown,draw=red,opacity=.3,very thin,line join=round] 
   (-1, 1, 0) -- 
   (-1 , 3, -2) --
   (1, -1, 0) --
   (1, -3, 2) --cycle ;

  \draw[fill=brown,draw=red,opacity=.3,very thin,line join=round] 
   (-2, 2, 0) -- 
   (-2 , 4, -2) --
   (0, 0, 0) --
   (0, -2, 2) --cycle ;
   \draw[fill=brown,draw=red,opacity=.3,very thin,line join=round] 
   (0, 0, 0) -- 
   (0 , 2, -2) --
   (2, -2, 0) --
   (2, -4, 2) --cycle ;
   

\node at (4.5, 3.2, .5) {true propensity score model, $e(x)$};
\node at (4.5, -3.5, 0) {propensity score model $m_\beta(x)$};

\draw[thick,->] (0,0,0) -- (2.5, 2, 2) node[anchor=south]{$\beta$};
\end{tikzpicture}
    }
    \caption{The contour surfaces for the true propensity score model, $e(x)$, are spherical because they only depend on the norm $||x||$. 
    The contour surfaces of linear model $m_\beta(x)$ are hyperplanes oriented by vector $\beta$. The planes cut through all spheres and so
units that are classified together with respect to $m_\beta(x)$ may actually belong to disparate classes with respect to the true model $e(x)$. This dissimilarity between the two propensity score models is captured by the approximate similarity quantity, $\J{m_\beta}{e}$ (of Definition~\ref{def:J}), which 
here attains its minimum possible value.
}
    \label{fig:entangled}
\end{figure}

\newcommand{\Xidot}{X_{i\bullet}}
\subsection{General dyadic network model}
\label{sec:dyadic}
In this section we consider a more realistic network model where the probability of an edge between units $i$ and $j$ is 
equal to $\expit(a+bX_{ij})$, and $X_{ij}$ is not necessarily decomposable into additive or multiplicative sender and receiver effects \citep{hoffrafteryetal2002,hoff2013likelihoods,airoldi2008mixed}.
For simplicity, let's assume that dyadic covariates are scalar, so that  $X_{ij}\in\mathbb{R}$, for every $i, j$.
Define binary treatment $Z_i$ as having at least one new neighbor, so that:
$$E(Z_i |\Xidot)= e(\Xidot) = 
1 - \prod_j (1 - \expit(a + b X_{ij})),$$
where $\Xidot = (X_{i1}, \ldots, X_{iN})\in\mathbb{R}^N$ can be construed as the individual-level covariates.
%
Consider an alternative linear propensity score model:
$m(\Xidot)=\expit(\gamma +\delta \Xidot^\top\mathbf{1}/n)$
 where $\gamma$ and $\delta$ are scalar, $\mathbf{1}$ is a column of ones and
 $\Xidot^\top\mathbf{1}/n$ is the average 
 of dyadic covariates for every $i$. This model projects the dyadic covariates to a particular summary for each unit (here, the mean). Such projection is necessary to avoid overfitting; for example, 
 a model of the form $m(\Xidot) = \expit(a + b^\top \Xidot)$ would have to fit $N+1$ parameters with $N$ data points.

The gradients for the two models, namely true model $e$ and misspecified model $m$,  with respect to the individual-level covariates are given by
$$\nabla e(\Xidot) =h(\Xidot)b\begin{pmatrix}
		\expit(a+bX_{i1})\\
		\vdots\\
		\expit(a+bX_{iN})
	\end{pmatrix} = h(\Xidot)b \cdot \mathbf{expit(a+b \Xidot)},
	\text{ and }\nabla m(\Xidot) = g(\Xidot)\delta \mathbf{1},$$
where $h$ and $g$ are positive scalar-valued functions, $\mathbf{1}$ is a column vector of ones, and $\mathbf{expit(a+b\Xidot)}$ is a column vector of $\expit$ values of
$\Xidot$, calculated element-wise.
A straightforward application of Theorem~\ref{theorem:general} implies that:
$$
J[m,e] = | E(\cos(\nabla m(\Xidot),\nabla e(\Xidot)))| = \mid E(\cos(1,
	\mathbf{expit(a+bX)})\mid.
$$

From this result it is clear that if the dyadic covariates, $X_{ij}$, are independent and identically distributed, then $m$ yields the same subclassification as $e$, that is $\J{m}{e} = 1$.
This is because the homogeneity and independence of $X$s will
make the expectation of each element of $\mathbf{expit(a+bX)}$ 
be the same, and so proportional to the vector of ones.
However, individuals are unlikely to be marginally identical, with variability coming in the form of different means, say, to account for popularity and sociability, and variances to account for homophily~\citep{granovetter1983strength}. In such cases we expect the subclassification of the misspecified model to be different that the true model, leading to bias in causal estimation. We 
explore these issues in the experiments of Section~\ref{sec:experiments}.

\section{Simulated experiments}
\label{sec:experiments}
Here, we illustrate the inherent biases of classical propensity score methodology with network treatments. In two simulation studies we evaluate the two theoretical settings of Section~\ref{sec:theory}. In the first simulation we consider the setting of Section~\ref{sec:inner} where the network model depends on the inner product of covariates. The covariates are chosen so that nodes that appear close in covariate space---essentially what a classical propensity score model attempts to infer---are in fact very far in network space. To make this explicit, this simulation is kept artificially small. In the second simulation we consider the much more general setting of dyadic covariates in the network model as in Section~\ref{sec:dyadic}. We illustrate that as the number of units grows, the effects of ignoring the treatment entanglement grow as well, measured as squared error from the true treatment effect.

\subsection{Multiplicative covariates simulation}
\label{section:simulation}
Consider five units, each having a one-dimensional covariate $X_i\in\mathbb{R}$; the pre-treatment network $\Gm$ has no edges and the post-treatment network $\Gp = (g_{ij}^+)$ has a 
probability distribution such that the connection $g_{ij}^+$ between two units $i$ and $j$ is an independent Bernoulli:
\begin{align}
\label{eq:propensity_true}
P(g_{ij}^+ =1| \Gm, \X) \propto \exp(X_i X_j + 1.0).
\end{align}
As before, $e(l, X_i)$ denotes the probability that unit $i$ makes $l$ new connections in total. This is a case of the dyadic network model that was analyzed in Section~\ref{sec:dyadic}, where we showed that classical propensity score methodology can be severely biased, especially in cases where covariates have a symmetric distribution around zero.

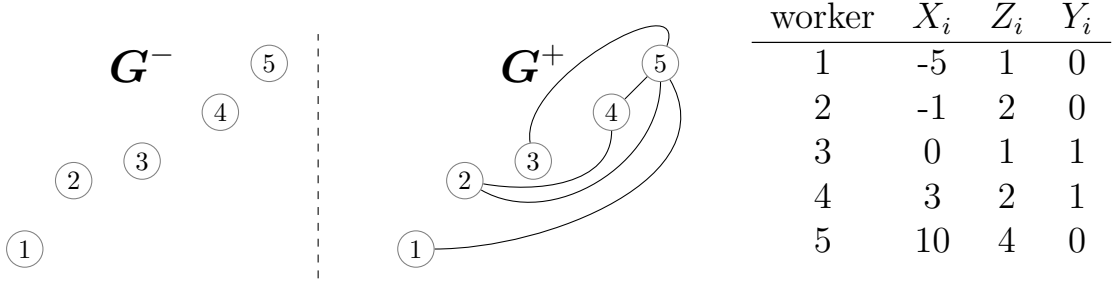
\begin{figure*}[!t]
\begin{subfigure}{.5\textwidth}
\begin{tikzpicture}[scale=1.3]
\drawunit{1}{0.3}{white}{1}
\drawunit{1.5}{1}{white}{2}
\drawunit{2.2}{1.2}{white}{3}
\drawunit{3}{1.7}{white}{4}
\drawunit{3.5}{2.2}{white}{5}

\draw [dashed] (4, 0) -- (4, 2.5);

\drawunit{4+1}{0.3}{white}{1}
\drawunit{4+1.5}{1}{white}{2}
\drawunit{4+2.2}{1.2}{white}{3}
\drawunit{4+3}{1.7}{white}{4}
\drawunit{4+3.5}{2.2}{white}{5}


\path (2) edge [out=-30, in=-90] (5);
\path (2) edge [out=-10, in=-90] (4);
\path (3) edge [out=100, in=70] (5);
\path (1) edge [out=0, in=-60] (5);
\draw (4)--(5);

\node at (2.2, 2.2) {\Large $\boldsymbol{G}^{-}$};
\node at (4+2.2, 2.2) {\Large $\boldsymbol{G}^{+}$};
\end{tikzpicture}
\end{subfigure}%
\begin{subfigure}{.5\textwidth}
\centering
\resizebox{140px}{!}{%
\begin{tabular}{cccc}
worker & $X_i$ & $Z_i$ & $Y_i$ \\
\hline
1      & -5   & 1  &   0 \\
2      & -1    & 2  &    0 \\
3      & 0     & 1  &  1   \\
4      & 3    & 2 & 1 \\
5     &  10 &  4 & 0
 \end{tabular}
}
\end{subfigure}
\caption{{\em Left:} The networks before ($\Gm$) and after ($\Gp$) 
the presumed intervention. {\em Right:} Observed data: $X_i$ is the 
covariate value for worker $i$; $Z_i$ is the treatment of $i$, i.e.,  
the number of new connections that worker $i$ made in $\Gp$; $Y_i$ denotes the outcome of worker $i$, where $Y_i=1$ if worker $i$ switched industries after the treatment period, and $Y_i=0$ otherwise.}
\label{figure:example}
\end{figure*}

Our goal is to use the data shown in Figure \ref{figure:example}
to estimate $\tau_2 = E(Y_i(2)-Y_i(1))$, i.e., the causal effect of making two new 
connections relative to making just one.
%
Let $\mathcal{S}(e)$ denote the set of units that have similar pairs of propensity scores $(e(1,X),e(2,X))$. 
The estimator will be of the form
\begin{align}
\label{eq:estimator}
\hspace{-.2cm}\hat{\tau}_{2}^e{\!\:{=}\!\:}\Ave(Y_i | i \in \mathcal{S}(e), Z_i=2){\!\:{-}\!\:}\Ave(Y_i | i \in \mathcal{S}(e), Z_i=1),
\end{align}
where $\Ave$ takes the average of outcomes for the units specified 
in its conditional statement; for instance $\Ave(Y_i|i\in \{1,2\})
= (1/2) (Y_1 + Y_2)$.
With this formulation, we compare a method that relies on the true propensity scores using model \eqref{eq:propensity_true} with the classical propensity score method that ignores the network structure, instead fitting a Poisson 
regression model. Based on the data of Figure~\ref{figure:example} 
the fitted propensity score model is given by:
\begin{align}
\label{eq:poisson}
m(l, X_i) = P(Z_i=l|X_i) & \propto \mathrm{Pois}(\lambda_i),\
\log \lambda_i  = 0.45 + 0.09 X_i,
\end{align}
where ``$\mathrm{Pois}$'' denotes the Poisson density, and 
the parameter estimates are rounded to two decimal points.
Table \ref{table:propensities} contains the estimated propensity scores for five units using Algorithm \ref{algo:ent} based on the true propensity score model in \eqref{eq:propensity_true} and on the incorrect propensity score model in \eqref{eq:poisson}. The tables also outline the subclassifications due to the two models, using the estimated propensities for $Z_i=1$ and $Z_i=2$: $\mathcal{S}(e)=\{1,2,4,5\}$ and $\mathcal{S}(m)=\{1,2,3,4\}$.  Unsurprisingly, the subclassifications lead to different estimates of the causal effect:
$\hat\tau_2^e = 0.5\text{ and }\hat\tau_2^m=0.$
In absolute value the bias is 0.5, which is substantial because the range of estimands is in [-1, 1] as outcomes are binary.

%
%
\begin{table}[!t]
\centering
\begin{subtable}{.45\linewidth}
\renewcommand*{\arraystretch}{1.1}
\hspace{-1cm}
\begin{tabular}{c|rrrrrr}
&  \multicolumn{5}{c}{propensity for $Z_i=\ldots$} \\
unit $(i)$ & 0 & 1 & 2 & 3 & 4 & \ldots\\ 
  \hline
1 & 0.00 & \marktopleft{c0}0.27 & 0.73 & 0.00 & 0.00 & \ldots \\ 

  2 & 0.00 & 0.24 & 0.67\markbottomright{c0} & 0.09 & 0.00& \ldots \\ 
 
  3 & 0.01 & 0.06 & 0.23 & 0.42 & 0.28 & \ldots \\ 

  4 & 0.00 & \marktopleft{c1}0.24 & 0.68 & 0.09 & 0.00 & \ldots \\ 
  5 & 0.00 & 0.27 & 0.73\markbottomright{c1} & 0.00 & 0.00 & \ldots
\end{tabular}
\end{subtable}
\begin{subtable}{.52\textwidth}
\renewcommand*{\arraystretch}{1.1}
\centering
\begin{tabular}{c|rrrrrrr}
& \multicolumn{7}{c}{propensity score for $Z_i=\ldots$} \\
unit $(i)$ & 0 & 1 & 2 & 3 & 4 & \ldots \\ 
  \hline
  1 & 0.37 & \marktopleft{c3}0.37 & 0.18 & 0.06 & 0.02 &  \ldots \\ 
  2 & 0.24 & 0.34 & 0.25 & 0.12 & 0.04 &  \ldots \\ 
  3 & 0.21 & 0.33 & 0.26 & 0.13 & 0.05 & \ldots \\ 
  4 & 0.13 & 0.26 & 0.27\markbottomright{c3} & 0.19 & 0.10 & \ldots\\ 
  5 & 0.02 & 0.08 & 0.15 & 0.20 & 0.20 &  \ldots \\ 
\end{tabular}
\end{subtable}
\caption{Propensity scores from two different models. 
The left panel is based on Algorithm \ref{algo:ent} using true model \eqref{eq:propensity_true}. The right panel is based on  Algorithm \ref{algo:ent} using the misspecified Poisson~\eqref{eq:poisson}. Units within dashed lines are subclassified together as having similar propensities to receive $Z_i=1$ and $Z_i=2$. The misspecified model leads to incorrect subclassification and, consequently, bias in causal inference.}
\label{table:propensities}
\end{table}

The explanation for such bias is straightforward. 
The graph $\Gp$ in the data of Figure~\ref{figure:example} is an atypical 
sample from its true distribution in Equation~\eqref{eq:propensity_true}. In particular, 
unit 3 is observed to have only one connection in $\Gp$. 
However, this unit has $X_3=0$, which implies that the probability that unit 3 connects to {\em any} unit is $e/(1+e)\approx0.73$, by model \eqref{eq:propensity_true}, and thus the unit should make $0.73 \cdot 4=2.92$ new connections in expectation; this can be verified by 
computing $\Ex{Z_3}$ in the left panel of Table~\ref{table:propensities}.
As mentioned earlier, the classical methodology estimates the 
propensity scores conditional on the observed post-treatment network $\Gp$, and thus underestimates the propensity scores for unit 3. 
Additionally, the large number of connections of unit 5 ($Z_5=4$) influences the Poisson model substantially, leading to the association of higher covariate $X$ values with a higher number of connections. This contributes to underestimating the propensity scores of unit 3 because $X_3=0<<10=X_5$.
This underestimation eventually leads to wrong subclassification and  biased estimates of the causal effect. 

\subsection{Large simulation studies}
The simulations in this section first consider the general setting of Section~\ref{sec:dyadic}, and then a generalization to probabilistic entanglement. Throughout this section, $\Gm$ has no edges and 
for $\Gp$ the probability of an edge is given by $P(g^+_{ij} =1) = \expit{(a_i/2+a_j/2+b X_{ij})}$, where $a_i$ can be thought of as an unobserved measure of  individual popularity (or sociability) while $X_{ij}\sim N(0,1)$ is some dyadic covariate. 

Specifically, we consider two settings where the covariates and network are symmetric ($X_{ij} = X_{ji}$ and $g^+_{ij} = g^+_{ji}$) and 
the treatment is either (1) ``making at least 1 new friend'' (that is $Z_i = 1_{\sum_j g^+_{ij} > 0}$), or (2) ``making more than 10 new friends'' (that is $Z_i = 1_{\sum_j g^+_{ij} > 10}$). In the first simulation we generate $a_i\overset{iid}{\sim}N(-5,\sigma^2)$ while in the second we generate $a_i\overset{iid}{\sim}N(-2,\sigma^2)$. In both settings, as $\sigma^2\rightarrow 0$, the probability of an edge approaches $\expit{(a + b X_{ij})}$ as in Section~\ref{sec:dyadic}. Finally, when we consider probabilistic entanglement, we also consider ``having at least 1 new friend'' but we do not enforce $X_{ij} = X_{ji}$ or $g_{ij}=g_{ji}$ and let $a_i\overset{iid}{\sim}N(-5,\sigma^2)$. 


To evaluate results we follow the procedure in \citep{dugoff2014generalizing} and simulate outcome data based on pre-treatment covariate information:
$Y_i(0) = 25 a_i + \epsilon_i,\quad Y_i(1) = Y_i(0) + 10$, where $\epsilon_i$ are independent normally distributed errors with standard deviation $\sigma$. The observed values for each simulation are $Y_{i} = Z_iY_i(1) + (1-Z_i)Y_i(0)$. In Tables~\ref{table:sym}-\ref{table:asym} we report the Root Mean Squared Error, $RMSE=\sqrt{\frac{1}{S}\sum (\widehat{ATE} - ATE)^2}$. We report these values when subclassifying on the true propensities and a misspecified model that estimates the propensity score via a logistic regression with $\sum_j X_{ij}$ as a covariate for unit $i$. Additionally, when studying the treatment ``at least 10 new friends'' we also demonstrate the substantial improvement over the marginal model by fitting a full network model, as proposed in Algorithm~\ref{algo:ent}. In fact we use an asymmetric and misspecified model to demonstrate that modeling the full network behavior is superior to only modeling the margin.
The results of the simulation demonstrate that misspecification via a marginal model can lead to very poor estimation of the ATE. 

\subsubsection{One new friend} 
\label{ssub:one_new_friend}

In Table~\ref{table:sym} we have the results of the first simulation where the network is symmetric and treatment is ``having at least 1 new friend''. Here, we only consider the true propensity score and the misspecified marginal logistic model. The results are reported over 5000 simulated networks. As expected, when the $a_i$ are very different ($\sigma^2$ is large), the misspecified model is unable to adapt to the true propensity model and so the {\rm RMSE} is very large. As $\sigma^2\rightarrow 0$ the discrepancy between the two approaches reduces.
\renewcommand{\arraystretch}{1.4}
\begin{table}[b!]
\centering
\begin{tabular}{ccc}
  \hline
  & \multicolumn{2}{c}{RMSE}\\
 $\sigma$& True model & Misspecified model \\
  \hline
  2.0  & 3.41 & 57.05 \\
  1.0    & 1.43 & 17.71 \\
  0.5  & 0.94 & 5.26 \\
  .25 & 0.77 & 1.88 \\
  .125  & 0.67 & 0.92 \\
  .0625 & 0.57 & 0.60 \\
  .03125 & 0.51 & 0.51 \\
   \hline
\end{tabular}
\caption{RMSE for $N=100$, $5000$ simulations, true model is $\expit{(a_i/2+a_j/2+X_{ij})}$; $a_i$ are randomly generated; $\sigma = \mathrm{SD}(a_i)$ and controls variation in individual propensity scores.
\label{table:sym}}
\end{table}

The experimental results agree with the theory of Section~\ref{sec:dyadic} from where it follows that subclassification similarity between the true model and the misspecified linear model 
is given by $|E(\cos(\mathbf{1}, \mathbf{expit(} a_i/2 + \mathbf{a/2 + b \Xidot)})|$, where $\mathbf{expit(} a_i/2 \cdot \mathbf{1} + \mathbf{a/2 + b \Xidot)})$ is a $N$-component vector with $\expit(a_i/2 + a_j/2 + b X_{ij})$ as its $j$-th element.
Consequently, when $\sigma=0$ and so $a_i=a_j$, the similarity is equal to one because 
$E(\expit(a + b X_{ij})) = E(\expit(a + b X_{ij'}))$ for all $i, j, j'$, since $X_{ij}$ are 
independent and identically distributed. In this case, 
a simple linear model will do just fine is subclassifying units, and hence in estimating the causal effect. However, as $\sigma$ increases 
$\mathbf{expit(} a_i/2 \cdot \mathbf{1} + \mathbf{a/2 + b \Xidot)})$ becomes varied with respect to $i$, which leads to worse subclassification, and thus more biased estimation of the causal effect.

\subsubsection{Multiple new friends} 
\label{sec:multiple_friends}

In Table~\ref{table:long} we report results for the general setting where treatment is defined as ``having at least 10 new friends''. Here, we not only report the results using the true model and the misspecified linear model but also using a network model where the probability of an edge is a function of $X_{ij}$ and a random node effect (we note that this model is still misspecified). The results of the simulation do not differ from the previous ones, but, crucially, we see that fitting even a misspecified network model provides substantial improvement (more than two-fold for higher variance $a_i$) to the estimation of treatment effects when the treatments are entangled 
compared to the classical linear propensity score model.

\begin{table}[t!]
\centering
\begin{tabular}{cccc}
  \hline
  & \multicolumn{3}{c}{RMSE} \\
 $\sigma$  & True model  & Random-effect model & Misspecified model\\
  \hline
  2.0    & 6.6 & 39.9 &  79.1 \\
  1.0    & 2.5 & 25.6 &  35.1 \\
  0.5  & 1.3 & 11.9 &  13.3 \\
  .25  & 1.0 & 4.1  & 4.3 \\
  .125  & 0.8 & 1.5  & 1.5 \\
  0.0625 & 0.7 & 0.8  & 0.8 \\
   \hline
\end{tabular}
\caption{RMSE for $N=100$, $5000$ simulations; true model $\expit{(a_i/2+a_j/2+X_{ij})}$; $a_i$ are random. Random-effect model is a misspecified model for the network with a random effect for each unit.\label{table:long}}
\end{table}

\subsubsection{Probabilistic entanglement} 
\label{sec:prob_ent}

In Table~\ref{table:asym} we report the results for probabilistic entanglement. In this setting edges are directed and so it is possible for person $i$ to connect to person $j$ (and hence be treated) while person $j$ does not connect to person $i$ (and remains untreated). However, the specified model implies that $g^+_{ij}$ is correlated with $g^+_{ji}$ and so it is reasonable that the performance of a marginal misspecified model will be poor. Here, we report the {\rm RMSE} based on the true propensities and the misspecified model over 5000 simulated networks. We see behavior that is very similar to the symmetric version of this simulation---when $\sigma^2$ is large the true propensities perform substantially better than the misspecified model, but when $\sigma^2\rightarrow 0$, the models become indistinguishable. In this simulation there is essentially no entanglement for small $\sigma^2$ and so the misspecified model is in fact the correct one.

\begin{table}[t!]
\centering
\begin{tabular}{ccc}
  \hline
  & \multicolumn{2}{c}{{\rm RMSE}} \\
 $\sigma$& True model & Misspecified model \\
  \hline
 2.0    & 3.46 & 57.01 \\
 1.0    & 1.44 & 17.76 \\
 0.5  & 0.93 & 5.32 \\
 .25  & 0.79 & 1.92 \\
 .125  & 0.68 & 0.93 \\
 .0625 & 0.57 & 0.61 \\
 0.0312 & 0.50 & 0.51 \\
   \hline
\end{tabular}
\caption{RMSE for $N=100$, $5000$ smulations; true model $\expit{(a_i/2+a_j/2+X_{ij})}$; edges $g^+_{ij}$ and $g^+_{ji}$ are positively correlated (probabilistic entanglement).\label{table:asym}}
\end{table}



\section{Concluding remarks}
\label{sec:discussion}
In this paper, we introduce the notion of entanglement, and analyzed 
it when a network of connections between units is observed at two pre-specified time points, and the individual treatment is a function of the possibly endogenous network change.
%
%
We showed that classical propensity score methodology ignores such treatment entanglement leading to causal estimates that may be severely biased.
To characterize the bias we  used a novel practical approximation of 
similarity between propensity score models, which acts as a proxy of how similar 
the subclassification-based causal estimates are from the two models. 
By focusing on well-known dyadic and inner-product network models describing the nature of treatment entanglement, we showed analytically that the similarity between the classical linear propensity score and the true propensity score can be arbitrarily bad, leading to invalid estimates of the causal effects from the classical methodology.~To mitigate such problems we proposed a methodology where we 
take into account the network evolution during the treatment period, 
over which we marginalize to calculate the correct individual propensity scores.
We illustrated our methodology on several simulated settings, which confirmed our theoretical findings.

It is important to note that our approach is not at odds with propensity score methodology, bur rather aims at fixing its problems when there is treatment entanglement. Algorithm~\ref{algo:ent} summarized our proposed approach to estimate individual propensity scores correctly under treatment entanglement, leaving the rest of the propensity score methodology unchanged. The architecture of our methodology is also modular so that  different definitions of treatment entanglement, and different model choices for network evolution could be accommodated; these issues were also discussed in Section~\ref{sec:methodology}.


We note that the problem of treatment entanglement is different than the problem of interference. As described in Section~\ref{sec:prelims}, 
interference is a statement on how an outcome of a unit may be affected by the treatment assignment of other units, whereas entanglement is a statement on how treatment assignment of a unit is affected by the treatment assignment of other units. 
Thus, entanglement 
and interference could both be present in a problem. For example, 
the labor outcomes of a professional worker could depend on the number of professional connections made by other workers. This is a problem where both entanglement and interference are present, which we will address in future work.

We also note that an additional issue with classical propensity score methodology relates to selection on unobservables, where 
self-selection in treatment depends on unobserved factors. In observational studies there is a concern that covariates 
$X$ that are important to explain the self-selection into treatment are missing, and thus the propensity scores cannot be appropriately estimated. For instance, it could be that more socialable workers make more professional connections in our data sample.
This issue is, however, largely orthogonal to the problem of treatment entanglement. 
In this paper, we focused on treatment entanglement because it presents methodological challenges even when all important factors are fully observed. Furthermore, the issue of entanglement has remained unnoticed despite related prior work in causality on networks, and an increasing interest in such problems.



\bibliographystyle{apalike}
\bibliography{main}

\newtheorem*{theorem*}{Theorem}

\appendix
\section{Proof of Theorem 4.3}
\begin{theorem*}
Let $e(x)$ be the unknown true propensity score model.
Let $m_\beta(x) = h(x^\top\beta)$, which may be misspecified, with $h : \mathbb{R} \to \mathbb{R}$ being a differentiable
monotone function.
For a vector-valued function, $g$, of covariates, let $\ngrad{g}$ denote the expected normalized gradient, that is,   $\ngrad{g}= \Ex{\nabla g(X_i) / || \nabla g(X_i)||}$.
Then,
\begin{align}
\J{m_\beta}{e} = \left |\ngrad{m_\beta} ^\top \ngrad{e} \right |=
\frac{|\beta^\top \ngrad{e}|}{||\beta||}.
\end{align}
If treatment is binary, then $\beta$ and $e(x)$ are connected through the moment equation:
\begin{align}
\label{eq:theorem_general2}
\Ex{\frac{e(X_i)}{h(X_i^\top\beta)} \frac{h'(X_i^\top\beta)}{1-h(X_i^\top\beta)} X_i} = \Ex{\frac{h'(X_i^\top\beta)}{1-h(X_i^\top\beta)} X_i}.
\end{align}

\end{theorem*}
\begin{proof}
By definition of model $m_\beta$, for any vector $w$ it holds that
 $$\cos(\nabla_x m_\beta(x), w) = \cos(h'(x^\top\beta) \beta, w) = 
 \pm\cos(\beta, w),$$
since $h$ is monotone. It follows that
\begin{align}
|\Ex{\cos(\nabla m_\beta(X), \nabla e(X))}|
=  |\Ex{\cos(\beta, \nabla e(X))}|
= |\Ex{\frac{\beta^\top \nabla e(X)}{||\beta|| ||\nabla e(X)||}}|
= \frac{|\beta^\top \ngrad{e}|}{||\beta||}.\nonumber
\end{align}
Note that $\ngrad{m_\beta} = \beta/||\beta||$ since $\nabla_x m_\beta(x)$ is collinear 
with $\beta$, as shown earlier.
The second part of the theorem follows from standard theory of statistical inference. Since treatment is binary $P(Z_i=1; X_i, \beta) = h(X_i^\top\beta)^{Z_i} (1-h(X_i^\top\beta))^{1-Z_i}$.
Let $h_i = h(X_i^\top\beta)$, then we can  write the log likelihood for all units as follows:
$\ell(\beta; \X, Z) = \sum_i Z_i \log h_i+ \sum_i (1-Z_i) \log(1 - h_i)$.
Since $\nabla h_i = h'(X_i^\top\beta) X_i \equiv h_i' X_i$, 
the gradient of the log likelihood (with respect to $\beta)$ is equal to
$\nabla \ell = \sum_i Z_i \frac{h_i'}{h_i} X_i   - \sum_i(1-Z_i) h_i' / (1-h_i) X_i$.
For the maximum likelihood equation we take the derivative to obtain 
\begin{align}
\label{eq953}
\sum_i Z_i [h_i'/h_i + h_i'/(1-h_i] X_i  & = \sum_i h_i' / (1-h_i) X_i\nonumber\\
\Ex{\frac{Z_i}{h(X_i^\top\beta)} \frac{h'(X_i^\top\beta)}{1-h(X_i^\top\beta)} X_i} & = \Ex{\frac{h'(X_i^\top\beta)}{1-h(X_i^\top\beta)} X_i} \nonumber\\
\Ex{\frac{e(X_i)}{h(X_i^\top\beta)} \frac{h'(X_i^\top\beta)}{1-h(X_i^\top\beta)} X_i} & = \Ex{\frac{h'(X_i^\top\beta)}{1-h(X_i^\top\beta)} X_i}.
\end{align}
Equation~\eqref{eq953} may be considered with respect to the empirical distribution of $X_i$ in the sample if the analysis is in the sample, or with respect to the superpopulation distribution of units if the analysis is more general. When $h = \expit$ then 
$h' = h (1-h)$ and the equation simplifies to 
$\Ex{e(X_i) X_i} = \Ex{h(X_i^\top\beta) X_i}$.
\end{proof}

\section{Result for $E(\expit(a + \sigma Z))$ of Section~\ref{sec:inner}}
\label{appendix:approx}
Here we analyze the function 
$$r(a, \sigma) = E(\expit(a + \sigma Z)) = E(\frac{e^{a + \sigma Z}}{1 + e^{a + \sigma Z}}),
$$
where $Z\sim\mathcal{N}(0, 1)$ is standard normal. 
We show that $\partial r/\partial \sigma$ has the opposite sign of $a$; 
function $r$ is monotone increasing with respect to $\sigma$ when $a<0$ and it is monotone 
decreasing with respect to $\sigma$ when $a > 0$. This result is sufficient for the theory to work.
First, we take the partial derivative:
\begin{align}
\label{app:eq1}
\frac{\partial r}{\partial \sigma} 
= \int_{-\infty}^{+\infty} \frac{e^{a + \sigma Z}}{(1 + e^{a + \sigma Z})^2}
   z \phi(z) dz,
\end{align}
where $\phi$ is the density of the standard normal. 
Suppose that $a > 0$. We claim that for every $u\ge 0$:
\begin{align}
\label{app:eq2}
g(a, -u) = \frac{e^{a - u}}{(1 + e^{a - u})^2} > 
  \frac{e^{a + u}}{(1 + e^{a + u})^2} = g(a, u).
  \end{align}
To prove this, direct algebraic manipulation of Eq.~\eqref{app:eq2} yields
$
1 + e^{a + u} > e^a + e^u,
$
which indeed holds for every $a>0$ and $u\ge 0$. From Eq.~\eqref{app:eq1} we obtain
$$
\frac{\partial r}{\partial \sigma} 
= \int_{0}^{+\infty} [g(a, \sigma z) - g(a, -\sigma z)] z \phi(z) dz
< 0,
$$
where the inequality follows from Eq.~\eqref{app:eq2}. The case where $a < 0$ is symmetric.
\end{document}

%% file: macros.tex
\usepackage{color}
\usepackage{tikz}
  \usetikzlibrary{arrows}
\usepackage{graphicx}
\usepackage{subcaption}

\newcommand{\Gp}{G^+}
\newcommand{\Gm}{G^-}

\newcommand{\expit}{{\text{expit}}}

 \newcommand{\drawunit}[4]{
\draw (#1, #2)  
node[draw,shape=circle, color=gray, text=black, fill=#3, inner sep=0.7mm] 
		(#4) {\footnotesize #4};
}

\newcommand{\Ex}[1]{E(#1)}
\newcommand{\X}{\boldsymbol{X}}

%% file: abstract.tex
\begin{abstract}
In experimental design and causal inference, it may happen that the treatment is not defined on individual experimental units, but rather on pairs or, more generally, on groups of units. For example, teachers may choose pairs of students who do not know each other to teach a new curriculum; regulators might allow or disallow merging of firms, and biologists may introduce or inhibit interactions between genes or proteins. In this paper, we formalize this experimental setting, and we refer to the individual treatments in such setting as {\em entangled treatments}. We then consider the special case where individual treatments depend on a common population quantity, and develop theory and methodology to deal with this case. In our target applications, the common population quantity is a network, and the individual treatments are defined as functions of the change in the network between two specific time points. Our focus here is on estimating the causal effect of entangled treatments in observational studies where entangled treatments are endogenous and cannot be directly manipulated. When treatment cannot be manipulated, be it entangled or not, it is necessary to account for the treatment assignment mechanism to avoid selection bias, commonly through a propensity score methodology. In this paper, we quantify the extent to which classical propensity score methodology ignores treatment entanglement, and characterize the bias in the estimated causal effects. To characterize such bias we introduce a novel similarity function between propensity score models, and a practical approximation of it, which we use to quantify model misspecification of propensity scores due to entanglement. One solution to avoid the bias in the presence of entangled treatments is to model the change in the network, directly, and calculate an individual unit's propensity score by averaging treatment assignments over this change.

\vfill

\textbf{Keywords:} Causal inference, observational studies, propensity scores, treatment entanglement, misspecification, bias, network data.

\end{abstract}